\documentclass[reqno]{amsart}
\usepackage[centertags]{amsmath}
\usepackage{graphicx}
\usepackage{amsfonts}
\usepackage{amssymb}
\usepackage{amsthm}
\usepackage{newlfont}
\usepackage[top=1in, bottom=1in, left=1in, right=1in]{geometry}
\usepackage{mathrsfs}
\usepackage{dsfont}

\usepackage[pdfborder={0 0 0},
   pdftitle={CONSTRUCTING SUBDIVISION RULES FROM RATIONAL MAPS},
  pdfauthor={BRIAN RUSHTON}, pdftex, bookmarks=true, bookmarksnumbered = true]{hyperref}

\theoremstyle{definition}
\newtheorem{thm}{Theorem}
\theoremstyle{definition}

\theoremstyle{definition}

\theoremstyle{remark}

\theoremstyle{definition}

\newcommand{\parlengths}{\setlength{\parindent}{0pt}}
\begin{document}
\date{\today}

\pdfbookmark[1]{CONSTRUCTING SUBDIVISION RULES FROM ALTERNATING
LINKS}{user-title-page}

\title{CONSTRUCTING SUBDIVISION RULES FROM ALTERNATING LINKS}

\author{Brian Rushton}
\address{Department of Mathematics, Brigham Young University, Provo, UT 84602, USA}
\email{brirush@gmail.com}

\begin{abstract}
The study of geometric group theory has suggested several theorems
related to subdivision tilings that have a natural hyperbolic
structure.  However, few examples exist.  We construct subdivision
tilings for the complement of every nonsingular, prime alternating
link.  These tilings define a combinatorial space at infinity, similar to the space at infinity for word hyperbolic groups.
\end{abstract}

\maketitle\parlengths

\section{Introduction}
\label{SectionIntroduction} This work grew out of a conjecture by
Cannon that connects two definitions of a hyperbolic group.
Specifically, he conjectured that every word hyperbolic group with
a 2-sphere at infinity acts properly discontinuously and
cocompactly by isometries on hyperbolic 3-space.  Cannon and
Swenson \cite{hyperbolic} were able to reduce this problem to
showing that a certain subdivision rule associated with the
hyperbolic group is conformal.  If we can understand when a
subdivision rule is conformal, we can solve this problem.

Recall that a subdivision rule is essentially a set of planar
tiles with a combinatorial rule for each tile that subdivides it
into smaller tiles.  They are defined more rigorously and
discussed in detail in \cite{subdivision}.

Every hyperbolic 3-manifold group has a conformal subdivision rule
\cite{hyperbolic}.  Unfortunately, there are not many explicit
examples of subdivision rules associated with hyperbolic groups or
with 3-manifold groups. Our purpose is to give an explicit
subdivision rule for every prime alternating link complement.
While these manifolds are not compact, their study may give hints
on how to proceed.

These subdivision rules are interesting in their own right, as
they allow us to construct the universal cover of the manifolds
(and thus the Cayley graph of the fundamental group) in a directly
geometric way.  They provide linear recursions for calculating
growth functions, and help us visualize the long-term behavior of
an infinite group.  They are a geometric analogue of power series,
in that an infinitely complex object (such as an analytic function
or the space at infinity of a group) can be modelled as the limit
of simpler objects (such as polynomials or subdivisions of the
sphere).

As a final note, subdivision rules for alternating links seem
somehow to help us visually distinguish between Thurston's eight
model geometries. For instance, the subdivision rules for the Hopf
link, the trefoil, and the Borromean rings are clearly different,
with properties suggestive of their $\mathbb{E}^3$, $\mathbb{H}^2
\times \mathbb{R}$, and $\mathbb{H}^3$ geometries, respectively.
This is discussed in more depth in Section \ref{FutureWork}.

\section{The Complement of a Link} \label{Complement}
There is a well known decomposition of every non-split alternating
link complement into two ideal polyhedra (for instance see
\cite{complement}or \cite{CuspStructures}). The boundary of each
polyhedron has edges and faces corresponding to the alternating
diagram we begin with; different diagrams can give different
decompositions. We assume that all alternating diagrams are
reduced. A diagram is reduced when no circle in the plane
intersects the diagram transversely at a single crossing and
nowhere else \cite{Colin}.

One of the more useful properties of this decomposition is that
the faces can be oriented in a checkerboard fashion, so that a
clockwise face only borders counterclockwise faces, and
vice-versa.  In fact, this orientation can be used to describe the
gluing map between the polyhedra: each face on one polyhedron is
glued to the corresponding face on the other polyhedron after a
single twist, the direction of the twist agreeing with the
orientation.  See Figure \ref{Checkerboard}.
\begin{figure}
\begin{center}
\scalebox{.7}{\includegraphics{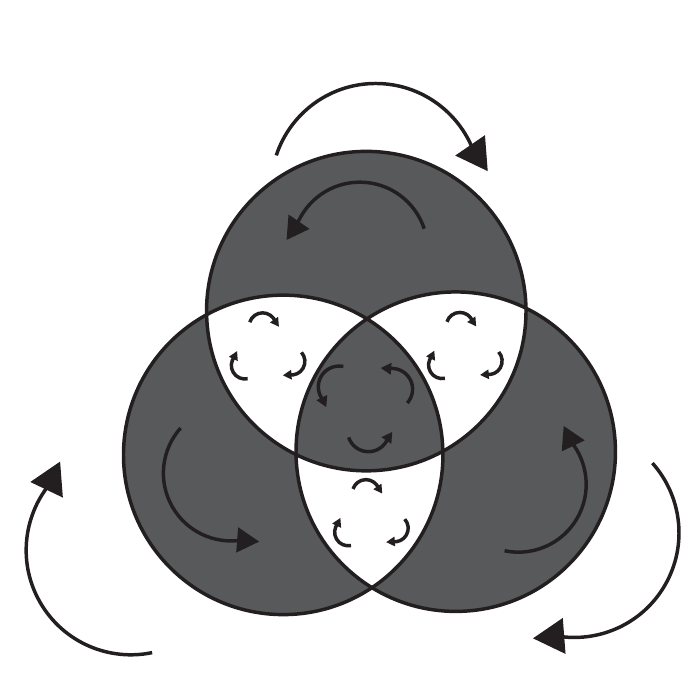}} \caption{The
checkerboard diagram for the Borromean rings.}
\label{Checkerboard}
\end{center}
\end{figure}

Also, we truncate the ideal polyhedra.  Thus, every ideal vertex
is replaced with a a new square face that is not identified with
any other face under the gluing map (see Figure
\ref{Boundaryinter}).
\begin{figure}[h]
\begin{center}
\scalebox{.5}{\includegraphics{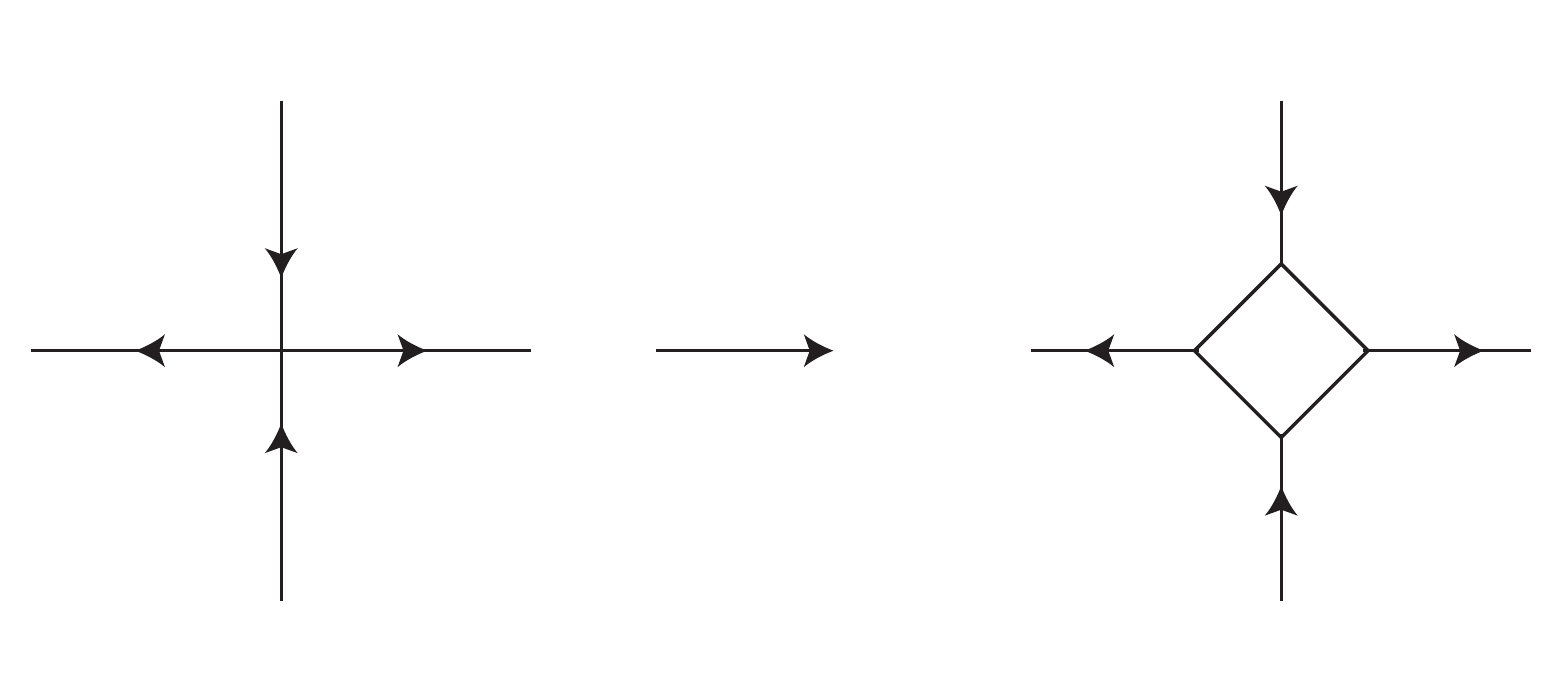}} \caption{Each
intersection is replaced by a square.} \label{Boundaryinter}
\end{center}
\end{figure}

When we glue these new polyhedra together, the squares form the
boundary of a compact manifold with boundary, namely the
complement of an open tubular neighborhood of the link.  We refer to these squares as `truncation squares'.

We mention one more property of an alternating link. Recall that a
link is split if there is a 2-sphere in its complement separating
one component of the link from another. A link is composite if
there is a 2-sphere that intersects the link exactly twice, and
the two pieces of the link thus separated are non-trivially
knotted. A link that is not composite is prime.  Menasco
\cite{split} showed that given a reduced alternating diagram of an
alternating link, the link is split if and only if the diagram is
split (i.e. has more than one component). Also, the link is prime
if and only if the reduced diagram is prime.  A diagram is prime
when no two regions share more than one edge.

\section{The Universal Cover and Replacement Rules}
\label{Universal} Our main goal is to study the space at infinity
of the fundamental groups of these link complements.  Although
these link complements all contain a copy of $\mathbb{Z}\times
\mathbb{Z}$ and are therefore not hyperbolic in the sense of
Gromov, they still have a space at infinity in a geometric sense.
More specifically, the universal cover will give us a sequence of
approximations to this space at infinity. We construct the
universal cover by starting with one half of the link complement,
gluing on copies truncated polyhedra recursively and looking at
the boundary of balls of constant word length. These boundary spheres, and the tilings on them, will approximate the space at infinity.  We illustrate the
following process using the Hopf link, shown in Figure \ref{HopfChecks}.  Note
that when we truncate vertices, an ideal polyhedron for the Hopf link is a cube.

\begin{figure}
\begin{center}
\scalebox{.70}{\includegraphics{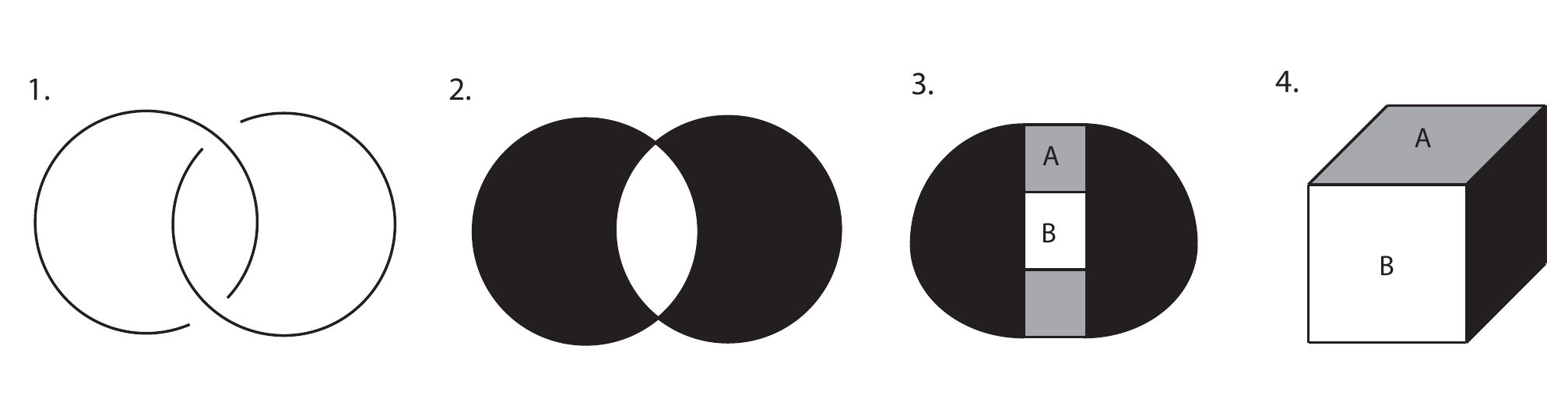}} \caption{An example of our truncation decomposition.  1)The Hopf Link. 2)The
checkerboard diagram for the Hopf link. 3)The truncated diagram for the Hopf link. 4)The polyhedron represented by the truncated diagram.} \label{HopfChecks}
\end{center}
\end{figure}

To get the recursive subdivision rule, we need to look at the
combinatorics of the gluing map.  In the original decomposition of
alternating links, each face is identified twice, each edge is
identified 4 times, and the vertices have been deleted, leaving a cusp, i.e. an ideal vertex.  When we
truncate ideal vertices, each edge on the new truncation square is identified twice,
and each new vertex four times.  These edges on the squares will be referred to as `truncation edges'.  All edges coming from the original diagram will be referred to as `link edges'.

As we build the universal cover, the covering map rules will cause
some collapsing. The universal cover is dual to the Cayley graph;
if we continue to glue a single polyhedron on each open face, we
would get the Cayley graph of a free group. These link complement
groups are not free, so the relators of these groups will cause
some collapsing as we build the Cayley graph. The dual notion to
relators is the local homeomorphism condition. Each edge in the
manifold locally meets four polyhedra, so each edge in the
universal cover must meet four polyhedra.

Our initial polyhedron will have as many faces and link edges as the
alternating diagram we began with, along with a truncation square for each
vertex of the alternating diagram. We glue a distinct polyhedron
onto each non-truncation face of the original polyhedron. So far no
other gluings are required by the local homeomorphism conditions.  This is shown in Figure \ref{HopfExplanation1}.

\begin{figure}
\begin{center}
\scalebox{.50}{\includegraphics{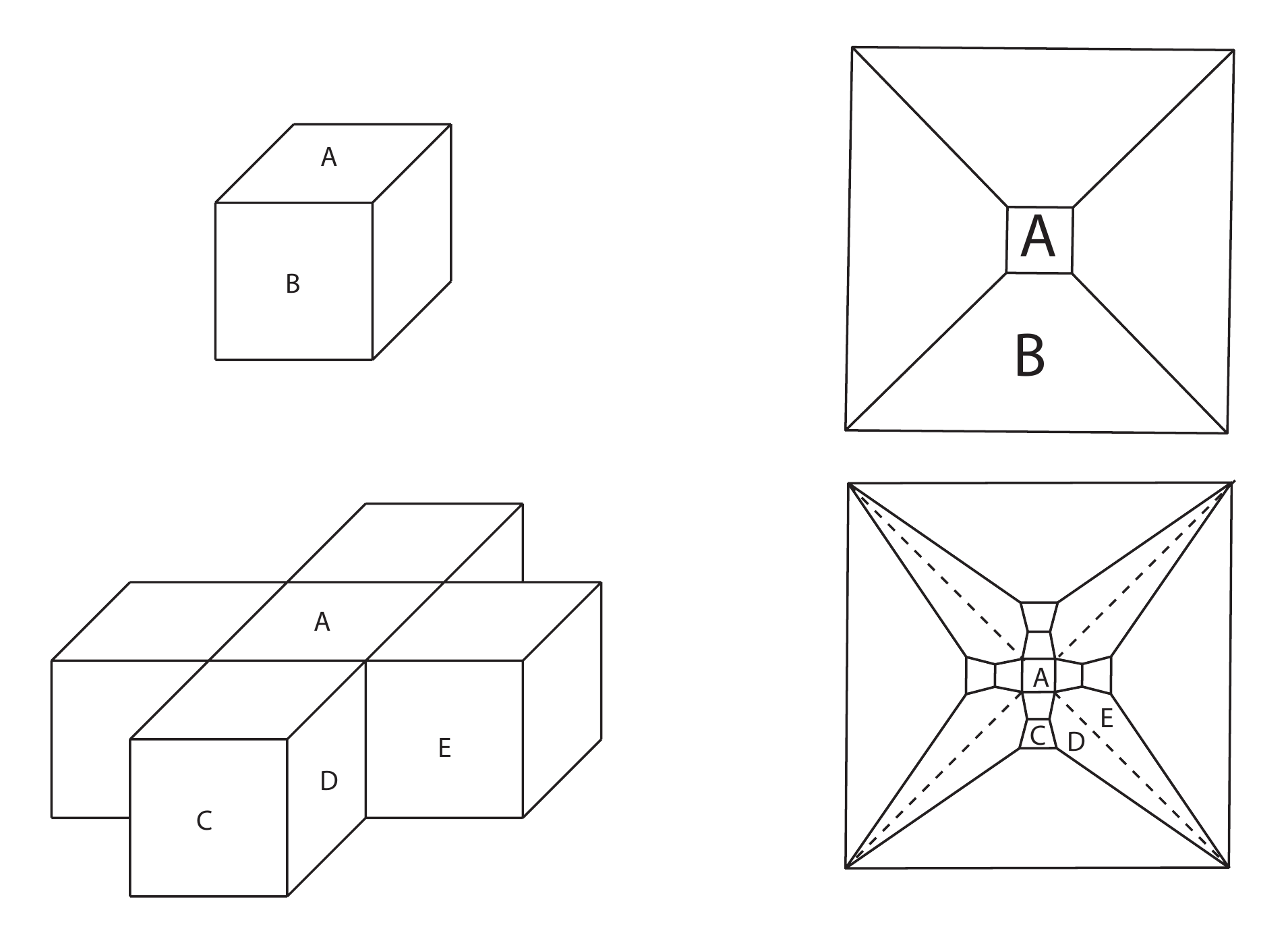}}
\caption{Constructing the universal cover of the Hopf link
complement. Under our truncation decomposition, the link
complement is the union of 2 cubes; their tops and bottoms are
truncation squares and are not identified.} \label{HopfExplanation1}
\end{center}
\end{figure}

Each link edge on the initial polyhedron has now had two
other edges identified to it, one from the face on either side.
These edges can only be identified one more time. Such edges will
be called `loaded'. Throughout this paper, loaded
edges are indicated by dotted lines.

We will refer to every face that is not a truncation square as a
region.  When we glue a half of the link complement onto a region
and project, the new regions we added inside of the old one will
be called subregions.  Those that touch the boundary of the old
region will be called edge subregions, and all others are interior
subregions.  In Figure \ref{HopfExplanation1}, the cube with faces
C and D was glued onto face B.  In this example, C is an interior
subregion and D is an edge subregion of the region B.

In the next stage, we again glue a polyhedron onto every open
face, but now we have `loaded pairs', two faces that share a common
loaded edge. We glue a single polyhedron onto both regions of a
loaded pair. See Figure \ref{HopfExplanation2}.  Again, all new regions are called edge subregions if
they touch the boundary of the loaded pair, and interior
subregions otherwise. In Figure \ref{HopfExplanation2}, every
subregion of the loaded pair consisting of D and E is an edge subregion.

\begin{figure}
\begin{center}
\scalebox{.50}{\includegraphics{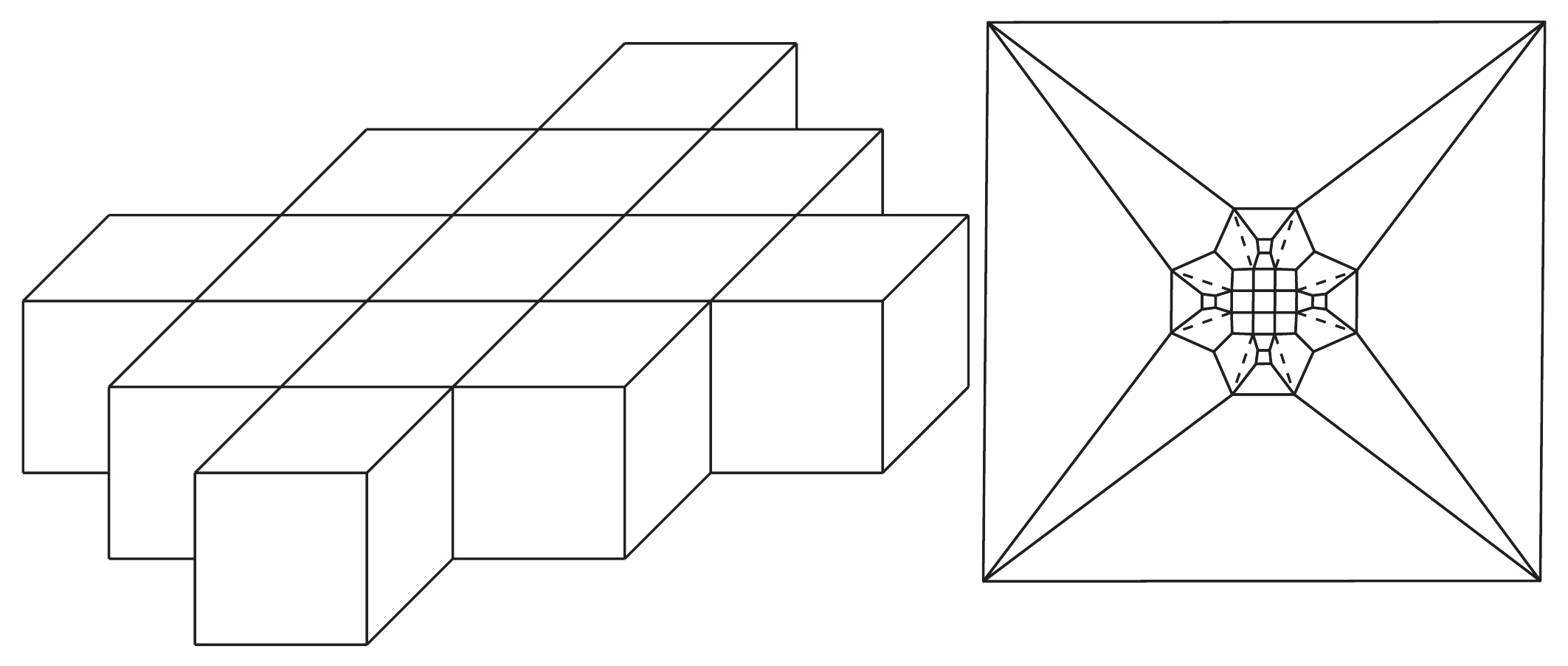}}
\caption{Constructing the universal cover of the Hopf link
complement. Under our truncation decomposition, the link
complement is the union of 2 cubes; their tops and bottoms are
truncation squares and are not identified.} \label{HopfExplanation2}
\end{center}
\end{figure}

The truncation edges are never considered loaded.  They
identify twice and are never covered up by newer pieces of the
link complement.

Note that in Figures \ref{HopfExplanation1} and \ref{HopfExplanation2} all visible faces fall
into three categories: truncation squares (like region A), regions
with no loaded edges (like regions B and C), and pairs of regions
sharing a loaded edge (like D and E). This is just one instance of
a more general phenomenon:

\begin{thm} \label{Loadedproof}
In each stage of constructing the universal cover of a prime,
alternating link, a region can have at most one loaded edge.
\end{thm}
\begin{proof}
This proof is inductive.  A subregion has loaded edges exactly when it touches the boundary of a region being subdivided.
Therefore, we need only show that no subregion touches more than one boundary edge.

Let R be a region with no loaded edges. Place a half of the link
complement on it. The interior subregions are free from loaded
edges, by definition. Each of the edge subregions touches the
boundary in at least one edge.  By primeness of the link, an edge
subregion cannot touch the boundary in more than one edge.  This follows from the last paragraph of Section \ref{Complement}. So the
replacement rule creates only subregions with one loaded edge or
none in this case.

Now let's pass to regions with loaded edges.  Since we place a single chunk
of the link complement on two such regions, consider a pair of
loaded regions at a time.  Let $R_1$ and $R_2$ be two regions that
share a loaded edge.  Then place a half of the link complement on
the two regions. In this case, $R_1$ and $R_2$ are treated as a unit, with a common boundary formed by the non-loaded edges of the two regions. This can be seen in Figure \ref{HopfExplanation2}, when a single cube covers up both D and E.  Assume a subregion $L$ has two
loaded edges. Then the edges it touches belong to $R_1$ or $R_2$.
If they both come from a single region, for instance, $R_1$, we
violate the primeness condition.  If it touches one edge of each,
then we get a contradiction due to orientation: $L$ must have an
orientation opposite of both $R_1$ and $R_2$, yet $R_1$ and $R_2$
have opposite orientations. Thus, even in the loaded case, only
subregions containing one or zero loaded edges are created.
\end{proof}
\section{Subdivisions With Boundary}
\label{SectionBoundary}

The patterns described in the last section are not exactly
subdivision rules.  A subdivision rule is defined by a set of closed polygonal tiles, with a rule for subdividing each tile into smaller tiles which are allowed to overlap on their boundaries.  If two tiles share an edge, then their subdivisions must agree on that edge, dividing it into the same number of pieces.  What we currently have are not subdivision rules, but replacement rules; i.e. some lines are
covered up and erased as we go along.  In a subdivision rule,
every line stays put forever, and we just draw over them, so that
every new stage is a refinement of the previous one.  This is what
makes subdivision rules powerful; the subdivisions approach a
limit of sorts.

To get a subdivision rule from a replacement rule, we can simply
homotope some of the fresh, unloaded edges to be drawn directly
over the loaded edges.  Thus, the loaded edges need not be erased.
Every loaded edge is bordered by two edge subregions, one on
either side. We can pick one of the two regions and merge it into
the other region by homotoping the non-loaded part of its boundary
onto the loaded edge as in Figure \ref{MovingLines}.

\begin{figure}[ht]
\begin{center}
\scalebox{.85}{\includegraphics{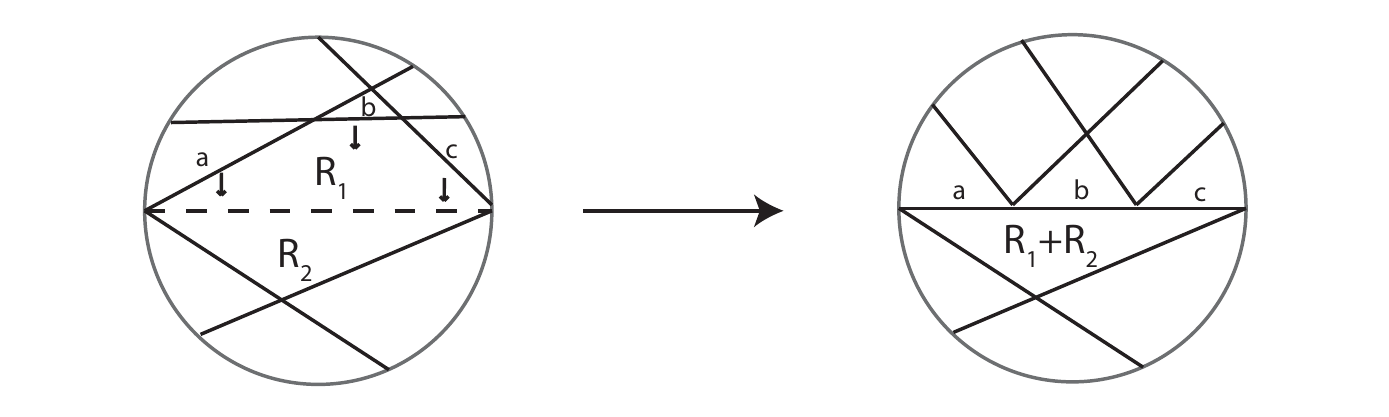}} \caption[How the
edges of a subregion move over]{This picture illustrates how the
lines of an edge subregion move to replace the loaded edge in a
general link. Lines a,b, and c all move over in this example.}
\label{MovingLines}
\end{center}
\end{figure}

If we can find an assignment of regions so that every loaded edge
is covered, we will have a subdivision rule.  We are now ready to
prove our main theorem.

\begin{thm} \label{BigTheorem}
Every prime, non-split alternating link admits a subdivision rule
with boundary.
\end{thm}
\begin{proof}
Fix an orientation on the regions of the truncated polyhedron associated to the alternating link.
Assign the non-loaded edges of every clockwise edge subregion at each
stage to replace the loaded edge. Since exactly one region on
either side of an edge is clockwise, this is well-defined at that
edge. Now, we have to worry if it is globally well-defined; in particular, if two
edge subregions that move happen share an edge or an interior vertex, then their
movement must agree on their common vertex or line.  And if a
subregion has more than one loaded edge, we can't use that
subregion's edges to replace the old loaded ones. These three potential problems are shown in Figure
\ref{BoundaryProblems}.

\begin{figure}
\begin{center}
\scalebox{.6}{\includegraphics{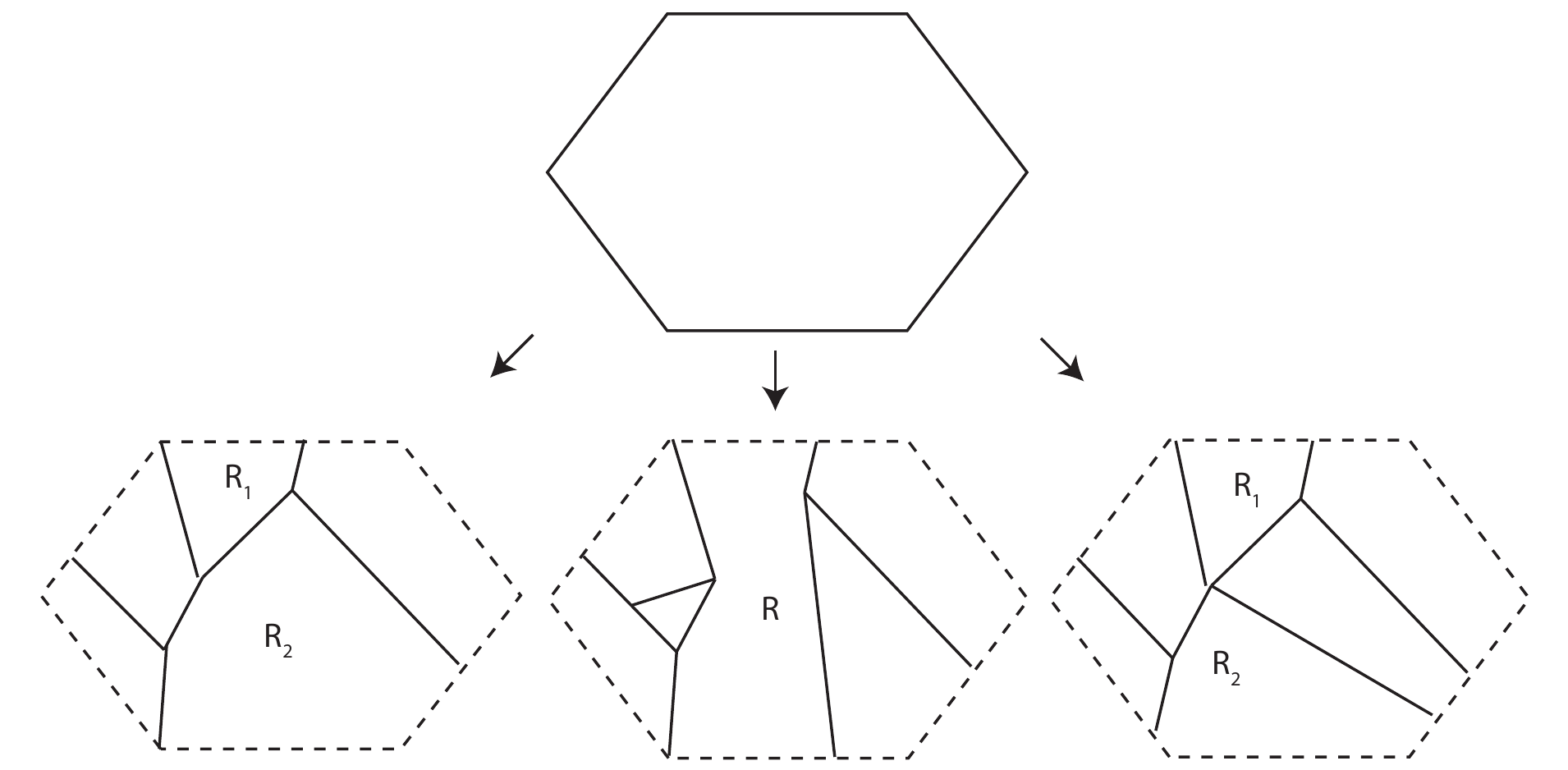}}\caption{These are examples of problems that might arise in a hypothetical hexagonal region.}
\label{BoundaryProblems}
\end{center}
\end{figure}

However, by Theorem \ref{Loadedproof}, no subregion has more than one loaded edge, so the last possibility does not occur.
We also need to eliminate the first two possibilities: moving regions sharing an edge or a vertex.  But two regions that share an edge must have opposite orientation.  Since we only move clockwise regions, no two moving regions share an edge.
Also, every interior vertex of a region has valence three.  Thus, if two subregions share an interior vertex, they share an edge.  This can be seen in the first two images in Figure \ref{BoundaryProblems}.  In particular, since no two moving regions share an edge, no two moving regions share an interior vertex.  So
our assignment is globally well-defined.

This gives a subdivision rule for every alternating link
complement.
\end{proof}

Note that we could equally well have chosen all counterclockwise
regions to move out.
  
Let's illustrate the process described in Theorem \ref{BigTheorem}, again using the Hopf link as our example.  Figure \ref{HopfToSub} shows the three possible replacement patterns we have.  Note that these are all views of the same polyhedron, a cube.  On the very right we have a loaded pair, which is a combination of a clockwise and counterclockwise region.  Two regions are hidden from view, but it is still a cube.  You can imagine the loaded edge as passing behind the diagram, connecting the leftmost and rightmost vertices.  The clockwise edge subregions in type B are pushed out according to our theorem.  The resulting subdivision rule divides the left and right edges into three pieces.  Since subdivision tilings must agree on the edges, type A or clockwise tiles must also be divided into three edges, since every clockwise region touches counterclockwise regions.  Thus, we add extra vertices to its left and right edges.  Those edge subregions become type C or loaded edges.  This is because they are `receiving' the clockwise regions being pushed out from type B tiles.  This is the process illustrated in Figure \ref{MovingLines}.  Note that in type C tiles, one half behaves as a clockwise region and one half behaves as a counterclockwise region.  This allows them to border both clockwise and counterclockwise regions.  This is typical of loaded pairs.

\begin{figure}
\begin{center}
\scalebox{.4}{\includegraphics{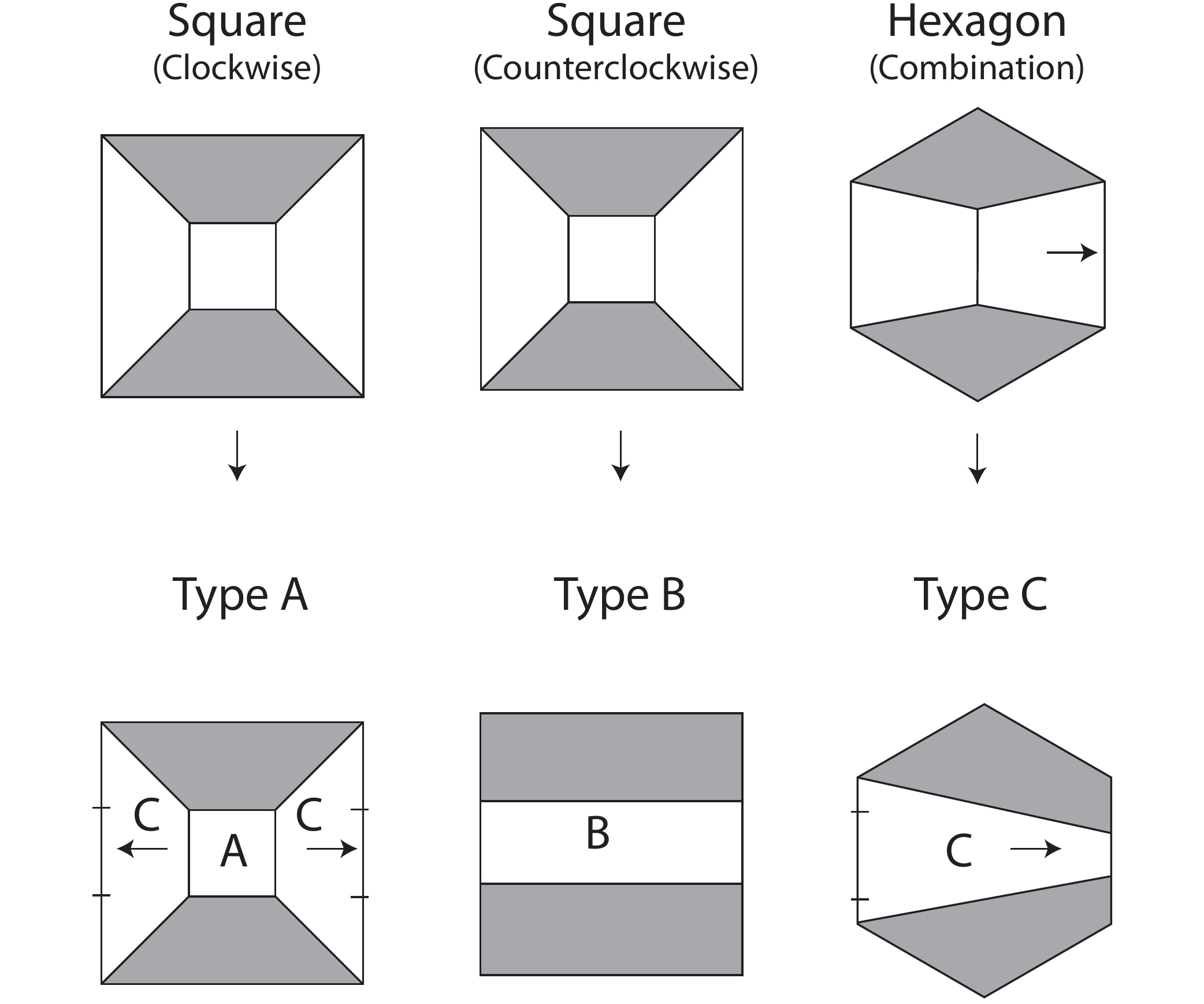}}\caption{The three kinds of regions that occur in the Hopf link's universal cover are converted into subdivision tiles.  the arrow on C tiles show how to orient them correctly.  In this figure and all following ones, the color grey represents a truncation square.  They are never subdivided.}
\label{HopfToSub}
\end{center}
\end{figure}

The first few subdivisions of tile A can be seen in Figure \ref{HopfBound}.  If you draw the subdivisions by hand, you will note differences between your results and Figure \ref{HopfBound}.  This is because the figures were drawn with Stephenson's Circlepack \cite{Circlepak}, along with some preprocessing programs by Floyd \cite{floyd}.  These programs maintain a subdivision tiling's combinatorial structure while shuffling things about to fit in the 'nicest' way.  Figures \ref{comparisons} and \ref{Bringb} were also created with these programs.

\begin{figure}
\begin{center}
\scalebox{.4}{\includegraphics{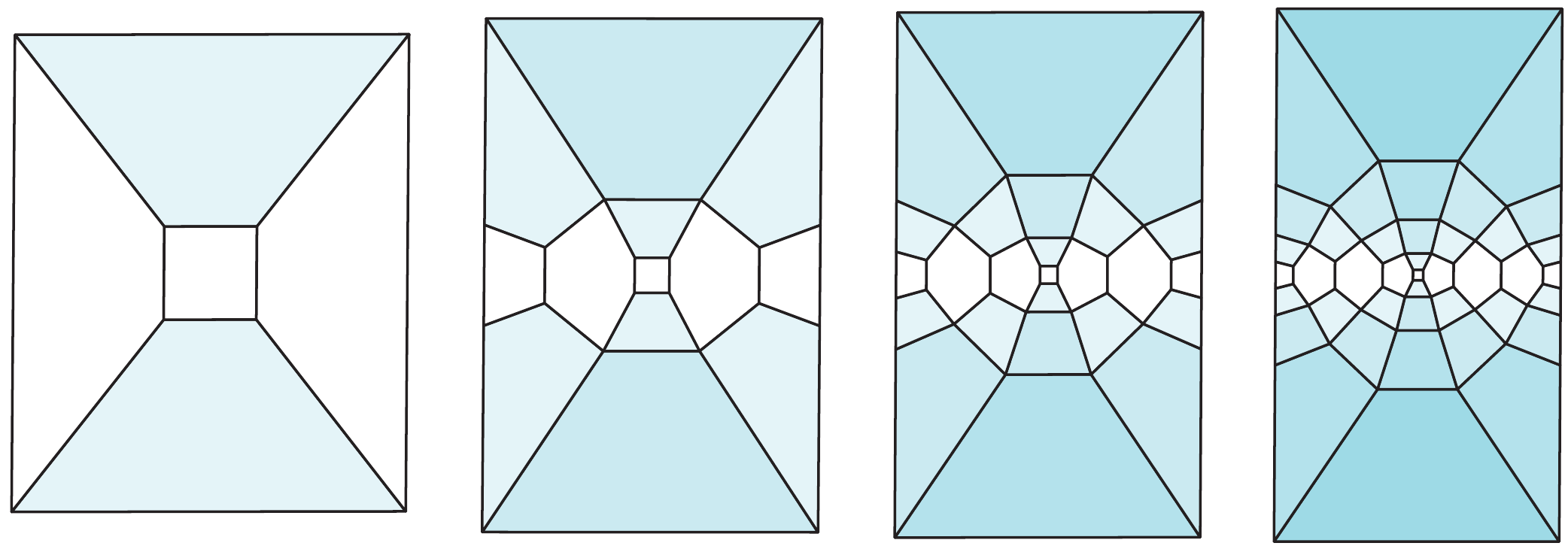}}\caption{The first four subdivisions of tile A.  Grey squares are truncation squares.  Compare with Figures \ref{HopfExplanation1} and \ref{HopfExplanation2}.}
\label{HopfBound}
\end{center}
\end{figure}

\section{Examples}
Let's look at some more examples. The trefoil is
the easiest example after the Hopf link. Its standard diagram can
be oriented so that all bigons are clockwise and all triangles are
counterclockwise.  There are three tile types for the trefoil
(four if you count the boundary squares). The replacement tiles
are shown in Figure \ref{TrefBoundary}, along with the subdivision
rule we get by moving all bigons out to replace loaded edges. Edge
labels are omitted and grey tiles represent boundary squares.  Note that the truncated polyhedron is a hexagonal prism.

\begin{figure}[ht]
\begin{center}
\scalebox{.6}{\includegraphics{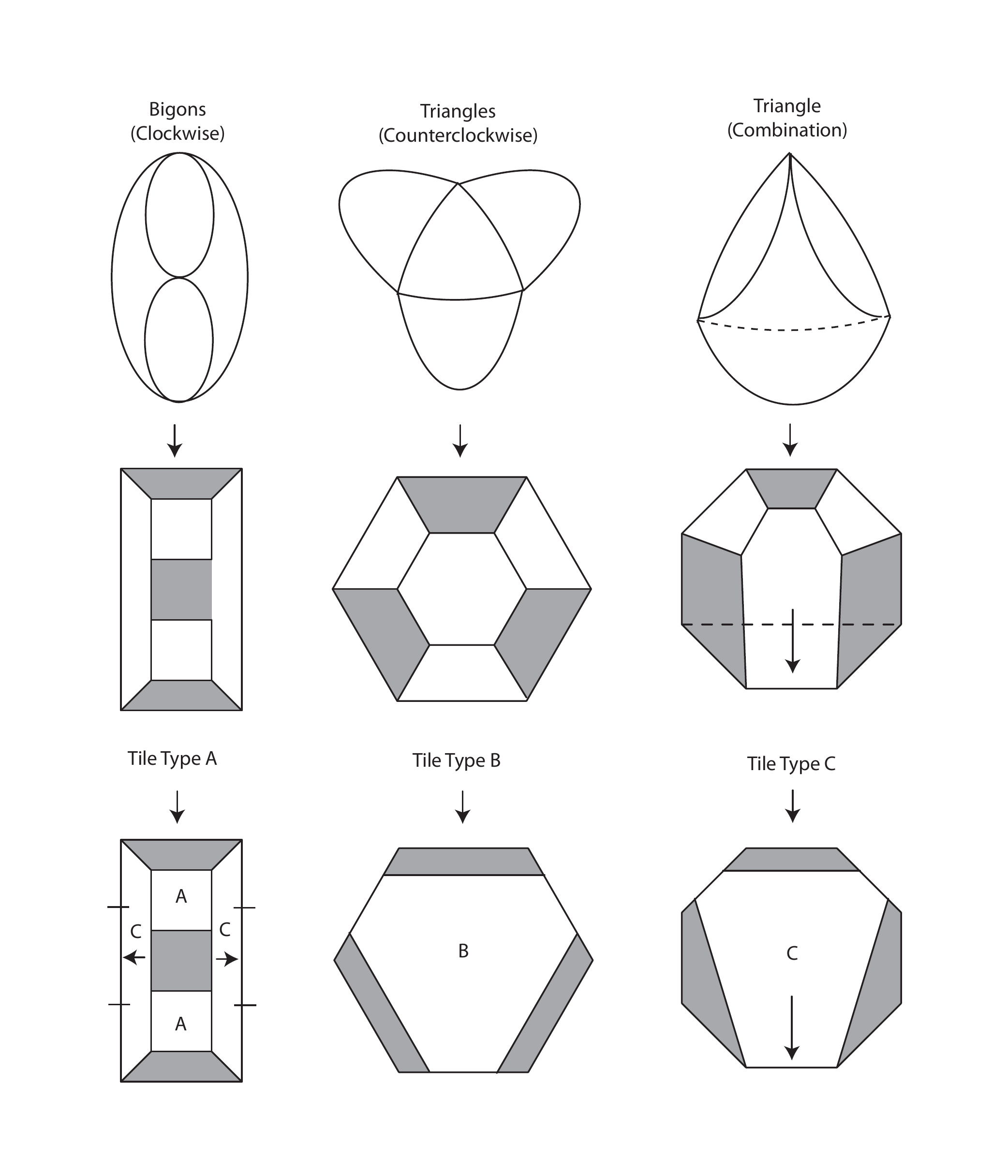}}\caption{The
replacement rules and subdivision rules for the trefoil.}
\label{TrefBoundary}
\end{center}
\end{figure}

Figure \ref{comparisons} shows the first few subdivisions of a
type A tile.

\begin{figure}
\begin{center}
\scalebox{.70}{\includegraphics{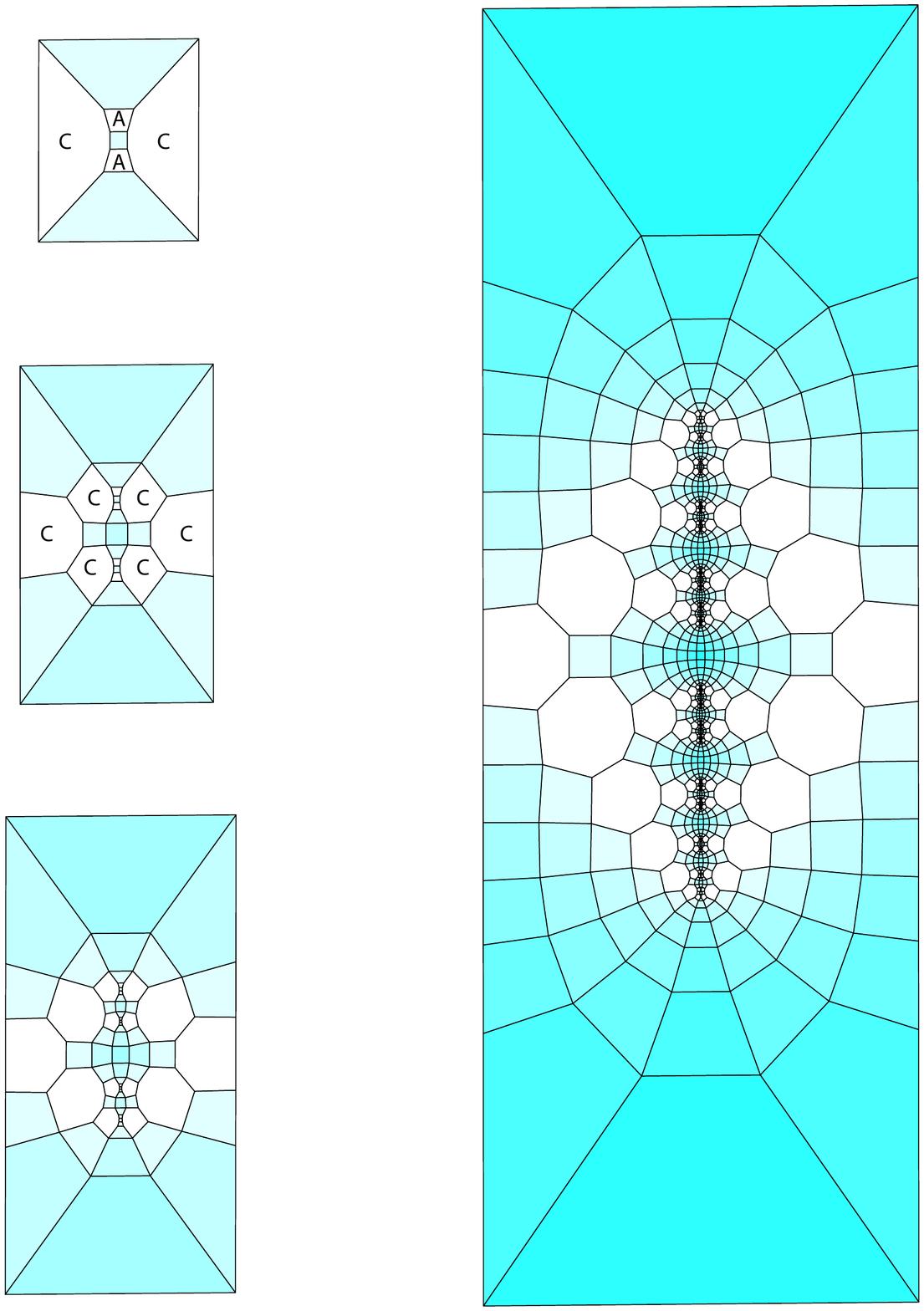}} \caption{The first
few subdivisions of tile A of the trefoil.  A few tile types have been indicated to aid in seeing the subdivision, as this is more complicated than the Hopf link.} \label{comparisons}
\end{center}
\end{figure}

Note that the boundary squares on the edge of the tiles should
always be placed to line up with the boundary squares of the
previously placed regions.  What's happening is that we're
building the universal cover of the torus at each boundary square.

Our last example is the Borromean rings.  The truncated polyhedron is a truncated octahedron.  Every region (i.e. faces that are not truncation squares) is identical.  However, it is impossible to get a subdivision rule that treats all hexagonal regions the same \cite{myself}. The subdivision rule is shown in \ref{BorroSubs}.

\begin{figure}
\begin{center}
\scalebox{.60}{\includegraphics{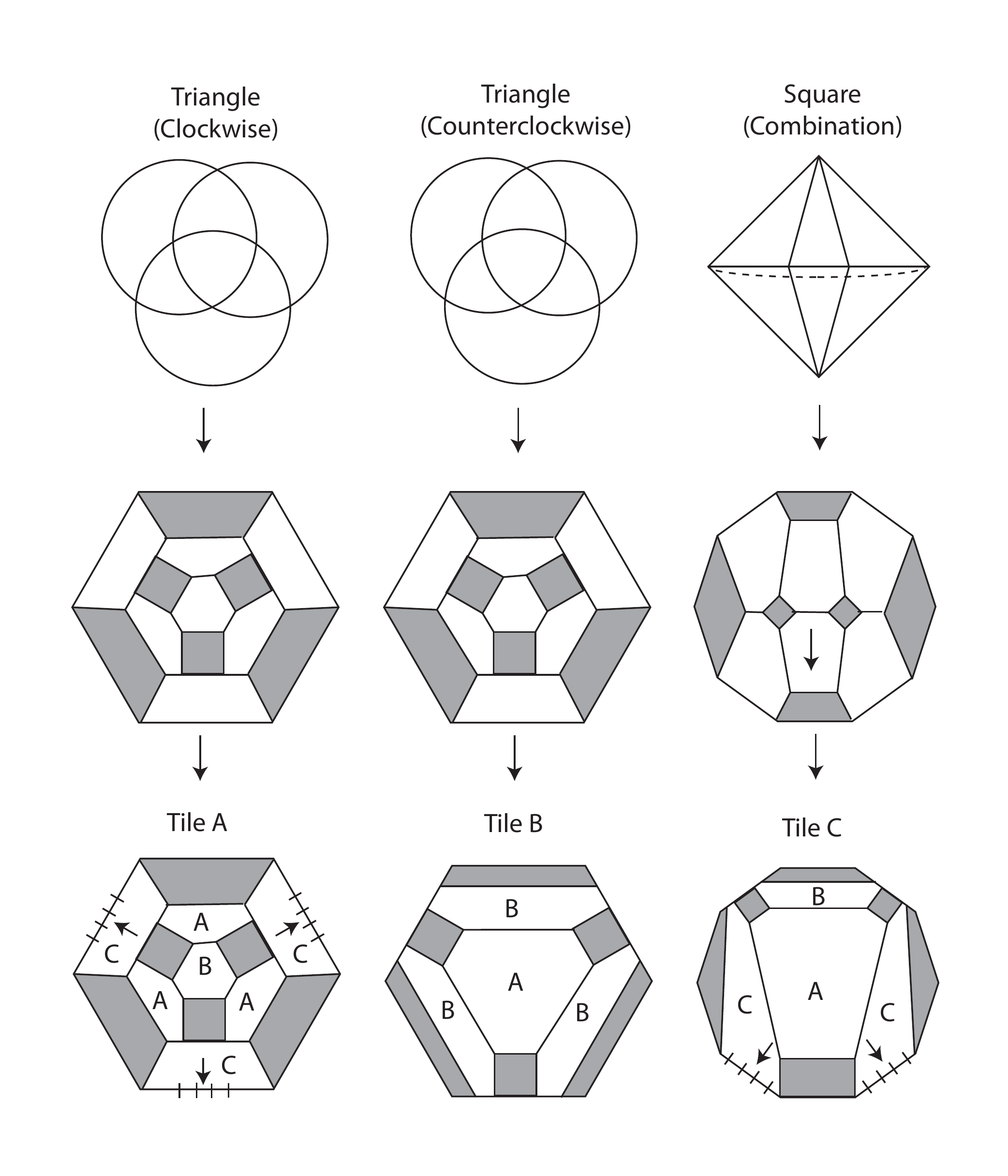}} \caption{The subdivision tiles for the Borromean rings.} \label{BorroSubs}
\end{center}
\end{figure}

Figure \ref{Bringb} contains the first few subdivisions of tile A
of the Borromean rings. As with the Hopf link and the trefoil, you
can see the universal cover of the torus growing at each vertex of
the link in a beautiful pattern.

\begin{figure}
\begin{center}
\scalebox{.75}{\includegraphics{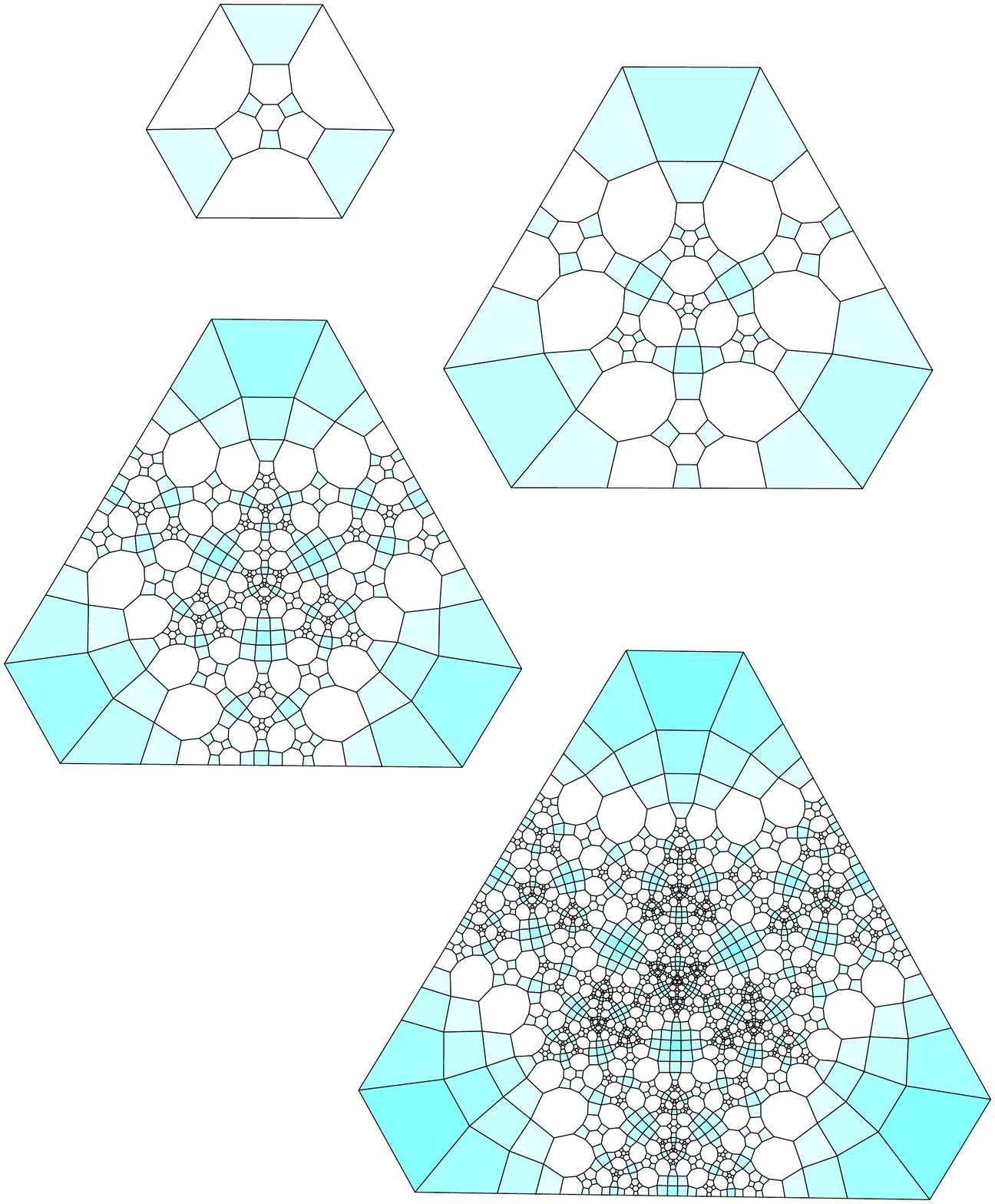}} \caption{The
first few subdivisions of tile A of the Borromean rings.}
\label{Bringb}
\end{center}
\end{figure}

\section{Future Work}
\label{FutureWork} The Hopf link is the only prime alternating
link with a Euclidean complement, and the trefoil is a
representative of other two-braid links, which have geometry
associated with $\mathbb{H}^2 \times \mathbb{R}$.  The Borromean
rings, and all other alternating links that are not two-braids,
have hyperbolic geometry \cite{split}.  This is reflected in their subdivision
rules.  The Hopf link has polynomial growth and only two boundary
planes.  The other two-braids have boundary planes that form a
circle of sorts; and the Borromean rings and all other alternating
links have a set of boundary planes that scatter all over the
sphere.

We have found subdivision rules for all surfaces cross the circle
and for all unit tangent bundles over surfaces (except for $P^3$,
which cannot have an infinite subdivision rule), and a single hyperbolic orbifold.  Their properties are related to
the open manifolds we described in the above paragraph. We hope to
make these connections more precise.

Also, these subdivision rules for links have the boundary tori
exposed. It would seem possible to alter these subdivision rules by Dehn
surgery, yet this has proved difficult in practice.  A few
examples have been worked out, but a general theory would be nice.

Finally, it would be natural to extend these subdivision rules to
non-alternating and/or composite links.  The main difficulty here
is that non-alternating links and composite links both tend to
have more collapsing than prime alternating links.  This truncated decomposition will give us two polyhedra as before, but
Theorem \ref{Loadedproof} no longer holds.  However, we have found subdivision
rules for all torus links \cite{myself}, and it is
easy to find subdivision rules for certain connected sums of Hopf
links, so there may be hope for a general theorem.

\newpage
\phantomsection

\bibliographystyle{plain}
\bibliography{bibpaper}

\begin{thebibliography}{1}

\bibitem{Colin}
C.~C. Adams.
\newblock {\em The Knot Book}.
\newblock American Mathematical Society, 2004.

\bibitem{CuspStructures}
I.~Aitchison, E.~Lumsden, and H.~Rubinstein.
\newblock Cusp structures of alternating links.
\newblock {\em Invent. math.}, 109:473--494, 1992.

\bibitem{subdivision}
J.~W. Cannon, W.~J. Floyd, and W.~R. Parry.
\newblock Finite subdivision rules.
\newblock {\em Conformal Geometry and Dynamics}, 5:153--196, 2001.

\bibitem{hyperbolic}
J.~W. Cannon and E.~L. Swenson.
\newblock Recognizing constant curvature discrete groups in dimension 3.
\newblock {\em Transactions of the American Mathematical Society},
  350(2):809--849, 1998.

\bibitem{floyd}
W.~J. Floyd.
\newblock tilepack.c, tilepackhistory.c, subdivide.c and subdividehistory.c.
\newblock Software, available from http://www.math.vt.edu/people/floyd/.

\bibitem{complement}
W.~Menasco.
\newblock Polyhedra representation of link complements.
\newblock In S.~J. Lomonaco, editor, {\em Low-dimensional Topology}, volume~20
  of {\em Contemporary Mathematics}, pages 305--325. American Mathematical
  Society, 1983.

\bibitem{split}
W.~Menasco.
\newblock Closed incompressible surfaces in alternating knot and link
  complements.
\newblock {\em Topology}, 23:37--44, 1984.

\bibitem{myself}
B.~Rushton.
\newblock Alternating links and subdivision rules.
\newblock Master's thesis, Brigham Young University, 2009.

\bibitem{Circlepak}
K.~Stephenson.
\newblock Circlepack.
\newblock Software, available from http://www.math.utk.edu/$\sim$kens.

\end{thebibliography}

\end{document}